\documentclass[12pt]{amsart}

\usepackage{amsthm}
\usepackage{enumerate}
\usepackage{amsmath}

\usepackage[colorlinks]{hyperref}

\newtheorem{thm}{Theorem}[section]
\newtheorem{lem}[thm]{Lemma}
\newtheorem{prop}[thm]{Proposition}
\newtheorem{cor}[thm]{Corollary}
\newtheorem*{nota}{Notation}

\title[Approximate identity and approximation properties]{Approximate identity and approximation properties of multidimensional Fourier algebras}
\author{Kanupriya}
\address{Kanupriya,\newline\indent Department of Mathematics,\newline\indent Indian Institute of Technology Delhi,\newline\indent New Delhi - 110016, India.}
\email{kanupriyawadhawan3@gmail.com}
\author{N. Shravan Kumar}
\address{N. Shravan Kumar,\newline\indent Department of Mathematics,\newline\indent Indian Institute of Technology Delhi,\newline\indent New Delhi - 110016, India.}
\email{shravankumar.nageswaran@gmail.com}

\begin{document}
\begin{abstract}
    For a locally compact group $G$, let $A^n(G)$ denote the multidimensional Fourier algebra given by $ \otimes_{n}^{eh} A(G).$ This work explores the approximation identity and operator amenability of the algebra $A^n(G)$. Further, we study the approximation properties (AP) and the concept of weak amenability of the multidimensional Fourier algebra. 
\end{abstract}

\maketitle

\section{Introduction}

The notion of amenability originated while studying groups and the related Banach algebras and has expanded to other areas such as functional analysis and operator algebras. Operator amenability is used in functional analysis, specifically while studying several operator algebras. This concept is useful in many areas of mathematics, including representation theory and non-commutative geometry. 

The Fourier algebra $A(G)$ of a locally compact group $G$ is an important subject of study in abstract harmonic analysis. It captures the group's algebraic structure using the Fourier transform. Further, the amenability properties of Fourier algebra provide information on the structure of the group and the associated operator algebras. Moreover, the operator amenability of $A(G)$ studies whether there are any bounded linear operators on $A(G)$ that behave as "amenable" in the sense of operator spaces.

A decade ago, the multidimensional Fourier algebra $A^n(G)$ appeared for a locally compact group $G$ in \cite{TT}, where the authors studied the multipliers of multidimensional Fourier algebra. The idea was motivated by the fact that for $n=2$ and $G$ is abelian, the Fourier transform of $L^1(G^2)$ corresponds to the space of bimeasures developed by Graham and Schreiber \cite{GS}. In this paper, we study the amenability and approximation properties of multidimensional Fourier algebra.

In 1972, B.E. Johnson \cite{J} established the notion of amenable Banach algebras. He proved that a locally compact group $G$ is amenable if and only if corresponding group algebra $L^1(G)$ is an amenable Banach algebra. Later in \cite{J2}, the author attempted to bring the concept of amenability into Fourier algebra, $A(G).$ However, the relationship between the amenability of $A(G)$ and $G$ was not adequate: there were compact groups $G$, such as $G = SO(3)$, for which $A(G)$ was not amenable. In the context of operator spaces, Z. J. Ruan \cite{Rua}  in 1995 formally introduced the concept of operator amenability of completely contractive Banach algebra. This discovery shed light on the link between group amenability and operator amenability of the Fourier algebra. The author proved that a group $G$ is amenable if and only if its Fourier algebra $A(G)$ is operator amenable. Further, the notion of operator amenability was studied for Fourier Stieltjes algebra by N. Spronk \cite{Sp}.

This paper focuses on examining the operator amenability of the multidimensional Fourier algebra and establishing a result similar to Z. J. Ruan's. Specifically, we show that a group G is amenable if and only if its multidimensional Fourier algebra $A^n(G)$ is operator amenable. The proofs are a consequence of some fundamental characteristics of tensor products.

As an extension of amenability, in 1989, Haagerup and Cowling \cite{HC} developed the notion of weak amenability for groups. If there is a net of finitely supported functions on a group G that uniformly approximates the constant function 1 on compact subsets of G, then the group is said to be weakly amenable. Further, in the next years, the notion of weak amenability is defined in terms of their geometric or algebraic characteristics. The study of the group's weak amenability in terms of multidimensional Fourier algebra is covered in Section 5 of this work.

A weaker notion than weak amenability is the existence of approximation properties (AP). The study of approximation properties of $A(G)$ explores whether $A(G)$ has certain 'nice' properties, such as being a dual space or having an approximate identity. Grothendieck \cite{Gr} established the foundation for studying the approximation properties of Banach spaces.

A key result by Eymard \cite{Eym1} shows that if G is an amenable group, then $A(G)$ has the approximation property (AP), i.e., $A(G)$ is a dual Banach space. The amenability of $G$ ensures that every bounded linear operator on $A(G)$ can be approximated by finite-rank operators. For operator spaces, Effros and Ruan \cite{ER1990} and for locally compact groups, Haagerup and Kraus \cite{HK} presented and analyzed other versions of Grothendieck's approximation properties. There is a close connection between these versions of the approximation property. In \cite{HK}, for instance, it was demonstrated that for a discrete group $G$ to possess the approximation property of Haagerup and Kraus, it must also possess the operator space approximation property of Effros and Ruan in the case of its reduced group $C^*$-algebra $C^{*}_{\rho}(G)$. Also, as stated by Junge and Ruan \cite{jun}, the approximation property is known to be true for discrete groups G if and only if their Fourier algebra $A(G)$ also satisfies the operator space approximation property. Further, T. Miao \cite{Miao2} studied operator space approximation properties of Fig\'{a}-Talamanca-Herz algebra $A_{p}(G)$. In Section 4 of the present study, we examine the approximation properties (AP) of the multidimensional Fourier algebra $A^n(G)$.

We now begin with some of the necessary preliminaries that are required in the sequel.

\section{Notations and Preliminaries}

    \subsection{Operator spaces and their tensor products}
        We first review some basics of operator spaces. As tensor products will play a major role, our main aim of this section is to collect the required preliminaries on these topics.

        Let $X$ be a linear space. The space of all $n\times n$ matrices with entries from the space $X$ will be denoted as $M_n(X).$ An {\it operator space} is a complex vector space $X$ together with an assignment of a norm $\|\cdot\|_n$ on the matrix space $M_n(X),$ for each $n\in\mathbb{N},$ such that
        \begin{enumerate}[(i)]
            \item $\|x\oplus y\|_{m+n}=max\{\|x\|_m,\|y\|_n\}$ and
            \item $\|\alpha x\beta\|_n\leq\|\alpha\|\|x\|_m\|\beta\|$
        \end{enumerate}
        for all $x\in M_m(X),$ $y\in M_n(X),$ $\alpha\in M_{n,m}$ and $\beta\in M_{m,n}.$ 

        Let $X$ and $Y$ be operator spaces and let $\varphi:X \rightarrow Y$ be a linear transformation. For any $n\in\mathbb{N},$ the {\it $n^{th}$-amplification} of $\varphi,$ denoted $\varphi_n,$ is defined as a linear transformation $\varphi_n:M_n(X) \rightarrow M_n(Y)$ given by $\varphi_n([x_{ij}]):=[\varphi(x_{ij})].$ The linear transformation $\varphi$  is said to be {\it completely bounded} if $\sup\{\|\varphi_n\|:n\in\mathbb{N}\}<\infty.$ We shall denote by $\mathcal{CB}(X,Y)$ the space of all completely bounded linear mappings from $X$ to $Y$ equipped with the norm, denoted $\|\cdot\|_{cb},$ $$\|\varphi\|_{cb}:=\sup\{\|\varphi_n\|:n\in\mathbb{N}\},\ \varphi\in\mathcal{CB}(X,Y).$$ We shall say that $\varphi$ is a {\it complete isometry (complete contraction)} if $\varphi_n$ is an isometry (a contraction) for each $n\in \mathbb{N}.$

        By Ruan's theorem, given an abstract operator space $X$ there exists a Hilbert space $\mathcal{H}$ and a closed subspace $Y\subseteq\mathcal{B}(\mathcal{H})$ such that $X$ and $Y$ are completely isometric. 

        For operator spaces $X$ and $Y$, {\it operator space Haagerup tensor norm} of $u\in M_n(X\otimes Y)$ is given by $$\|u\|_h=\inf\left\{ \|x\|\|y\|:u=x\odot y,x\in M_{n,r}(X), y\in M_{r,n}(Y),r\in\mathbb{N} \right\}.$$ 
        Here $x \odot y$ denotes the inner matrix product of $x$ and $y$ defined as $(x \odot y)_{i,j} = \underset{k=1}{\overset{r}{\sum}} x_{i,k} \otimes y_{k,j},$ where $x \odot y \in M_{n}(X \otimes Y).$ For more details, we refer to \cite[Chapter 9]{ER3}. The quantity $\|\cdot\|_h$ is an operator space norm and the resulting operator space will be denoted by $X\otimes^h Y$. In the case when $X$ and $Y$ are C*-algebras, for $n=1,$ the Haagerup norm can be written as follows. For $u\in X\otimes^h Y,$ $$\|u\|_h=\inf \left\{ \left\| \underset{n\in\mathbb{N}}{\sum} x_n x_n^\ast \right\|_X^{1/2} \left\| \underset{n\in\mathbb{N}}{\sum} y_n^\ast y_n \right\|_Y^{1/2} :u= \underset{n\in\mathbb{N}} {\sum}x_n\otimes y_n \right\}.$$ 
            
        The {\it extended Haagerup tensor product} of $X$ and $Y,$ denoted $X\otimes^{eh} Y$ is defined as the space of all normal multiplicatively bounded functionals on $X^\ast\times Y^\ast.$ By \cite{ER2}, $(X\otimes^h Y)^\ast$ and $X^\ast\otimes^{eh}Y^\ast$ are completely isometric. Given dual operator spaces $X^\ast$ and $Y^\ast,$ the {\it $\sigma$-Haagerup tensor product (or normal Haagerup tensor product)} is defined by $$X^\ast\otimes^{\sigma h}Y^\ast=(X\otimes^{eh}Y)^\ast.$$ Further, the following inclusions hold completely isometrically: $$X^\ast\otimes^h Y^\ast \hookrightarrow X^\ast\otimes^{eh} Y^\ast \hookrightarrow X^\ast\otimes^{\sigma h} Y^\ast.$$

       \begin{nota}
            The n-tensor product of a space X is denoted by $$\otimes_{n}X=\underbrace{ X \otimes X \otimes\cdots\otimes X}_n.$$ Now, we define $\otimes_{n}^h X,$ $\otimes_{n}^{eh} X,$ and $\otimes_{n}^{\sigma h} X$ as the completion of $\otimes_{n}X$ with respect to their corresponding norms.
        \end{nota}

        For more on tensor products, we refer to \cite{ER1, ER2}. For further details on operator spaces, the reader can refer to \cite{ER3} or \cite{Pis}.

\subsection{Fourier algebra and its multidimensional version}
    Let $G$ be a locally compact group and fix a left Haar measure $dx$ on $G$. Let $\lambda_G$ denotes the left regular representation of $G$ acting on Hilbert space $L^2(G)$ by left translations: $\lambda_G(s)f(t) = f(s^{-1}t)$ for some $s,t \in G$ and $f \in L^2(G).$ The group von Neumann algebra $VN(G)$ is the smallest self adjoint subalgebra in $\mathcal{B}(L^2(G))$ that contains all $\lambda_G(s)$ for all $s \in G$ and is closed in weak operator topology. The Fourier algebra $A(G)$ is actually the predual of $VN(G)$ and every function $u$ in $A(G)$ is represented by $u(x) = \langle \lambda_G(x)f,g \rangle$ with $\|u\|_{A(G)} = \|f\|_2\|g\|_2.$ The duality between $A(G)$ and $VN(G)$ is given by $\langle T,u \rangle = \langle Tf,g \rangle.$ For more details, one can refer to Eymard \cite{Eym1}. The group algebra $L^1(G)$ is an involutive Banach algebra and its enveloping algebra is denoted by $C^*(G)$. The closure of the left regular representation of $L^1(G)$ in the operator norm is known as reduced $C^*$-algebra which is denoted by $C_r^*(G)$. We shall denote the group  $G\times G\times\cdots\times G\ (n\mbox{-times})$ by $G^n$.
 
    The multidimensional version of Fourier algebra denoted by $A^n(G)$, is actually defined as collection of all functions $f\in L^{\infty}(G^n)$ such that a normal completely bounded multilinear functions $\Phi$ on $\underbrace{VN(G) \times VN(G) \times \cdots \times VN(G)}_n$ satisfying $f(x_1,x_2,...,x_n)= \Phi(\lambda(x_1),...,\lambda(x_n))$ From \cite{TT}, $A^n(G)$ coincides with $\otimes_n^{eh} A(G).$ Also, the dual of $A^n(G)$ is completely isometrically isomorphic to $VN^n(G):= \otimes_n^{\sigma h}VN(G)$. See \cite{TT} for more details.
 \subsection{Multiplier and Completely bounded Multiplier}
 Let $\mathcal{A}$ be a commutative regular semi-simple completely contractive Banach algebra with Gelfand spectrum $\Delta(\mathcal{A});$ consequently, $\mathcal{A}$ may be seen as a subalgebra of the algebra $C_0(\Delta(\mathcal{A}))$ of all continuous functions that vanish at infinity. A continuous function $b: \Delta(\mathcal{A}) \rightarrow \mathbb{C}$ is considered a multiplier of $\mathcal{A}$ if $b\mathcal{A} \subseteq \mathcal{A}$. This results in a well-defined bounded linear map $m_b$ on $\mathcal{A}$, defined as $m_b(a) = ba$.Further, if the map $m_b$ is completely bounded, then $b$ is considered a completely bounded multiplier. We denote $M\mathcal{A}$ (or $M^{cb}\mathcal{A}$) as the space of all multipliers (or completely bounded multipliers) of $\mathcal{A}$. A bounded linear map $T: \mathcal{A} \rightarrow \mathcal{A}$ is defined as $T = m_b$ for any $b \in M\mathcal{A}$ if and only if $T(x)y = xT(y)\ \forall\ x, y \in \mathcal{A}$ (see \cite[Proposition 2.2.16]{K1}).  If $b \in M^{cb}\mathcal{A}$, we denote the completely bounded norm of $m_b$ as $\|.\|_{cbm}$. The functions $b \in M^{cb}\mathcal{A}$ are commonly associated with the corresponding linear transformations $m_b$. To learn about multipliers of the Fourier algebra, one can refer to \cite{KL2}.
 
 \subsection{Amenability and approximation properties}
 A locally compact group $G$ is said to be amenable if there exists a $M^1(G)$- invariant mean $M$ on $L^{\infty}(G)$, i.e., $M(\mu \ast f) = M(f)\ \forall\ \mu \in M^1(G)$ and $f\in L^{\infty}(G),$ where $M^1(G)$ denotes the convex set formed by probability measures on $G.$ One can read about characterizations of amenable groups and Fourier algebra in \cite{P}.
 Let $\mathcal{A}$ be a completely contractive Banach algebra, and let $X$ be an operator space that is also a $\mathcal{A}$-module. If both the left and right module multiplication maps, denoted as $m_l: \mathcal{A} \otimes X \rightarrow X$ and $m_r: X \otimes \mathcal{A} \rightarrow X$, can be extended to complete contractions $m_l: \mathcal{A} \hat{\otimes} X \rightarrow X$ and $m_r: X \hat{\otimes} \mathcal{A} \rightarrow X$, then $X$ is referred to as a completely contractive $\mathcal{A}$-module. A derivation is a linear map $D: \mathcal{A} \rightarrow X$ which satisfies $D(ab) = a.D(b) + D(a).b\ \forall\ a, b \in \mathcal{A}$. It can be observed that $X^*$ is also a completely contractive $\mathcal{A}$-module if $X$ is one. $\mathcal{A}$ is called operator amenable if every completely bounded derivation $D : \mathcal{A} \rightarrow X$, where $X$ is a completely contractive $\mathcal{A}$-module, is inner (i.e. $D(a) = a.f -f.a$ for some $f \in X$). It is a well-known result that $A(G)$ is operator amenable if and only if $G$ is amenable, see \cite[Theorem 4.2.7]{KL2}. In section 4 of this article, we discuss the operator amenability of multidimensional Fourier algebra $A^n(G).$ To know more about the approximation property for locally compact groups, one can refer to \cite[Page-683]{HK}.

\section{Approximate identity and Amenability}

    In this section, we prove the existence of approximate identity. Further, we show some equivalent conditions related to amenability.

   Let us begin by proving the necessary and sufficient condition for the existence of approximate identity in $A^n(G)$.
    \begin{thm}\label{ABAI}
        The algebra $A^n(G)$ possesses a bounded approximate identity if and only if $G$ is amenable.
    \end{thm}
    \begin{proof}
        Suppose that $A^n(G)$ possesses a bounded approximate identity, say $(u_\alpha)_\alpha.$ For $t\in G,$ let $\xi_t$ denote the corresponding character on $A^n(G).$ Then the map $id\otimes\xi_t\otimes\ldots\otimes\xi_t:A^n(G)\rightarrow A(G)$ is a completely contractive homomorphism. Let $0\neq\phi\in A(G)$ and $t_0\in G$ such that $\phi(t_0)\neq 0.$ Note that 
        \begin{align*}
            & [(id\otimes\xi_t\otimes\ldots\otimes\xi_t)(u_\alpha)]\phi \\ =& [(id\otimes\xi_t\otimes\ldots\otimes\xi_t)(u_\alpha)(id\otimes\xi_t\otimes\ldots\otimes\xi_t)](\phi\otimes \phi(t)^{-1}\phi\otimes\ldots\otimes \phi(t)^{-1}\phi) \\ =& (id\otimes\xi_t\otimes\ldots\otimes\xi_t)(u_\alpha(\phi\otimes \phi(t)^{-1}\phi\otimes\ldots\otimes \phi(t)^{-1}\phi)) \\ \rightarrow & (id\otimes\xi_t\otimes\ldots\otimes\xi_t)(\phi\otimes \phi(t)^{-1}\phi\otimes\ldots\otimes \phi(t)^{-1}\phi) = \phi.
        \end{align*}
        This implies that $A(G)$ has a bounded approximate unit. Thus, $G$ is amenable by \cite[Pg. 96, Theorem 10.4]{P}.

        Conversely, suppose that $G$ is amenable. Then by \cite{Lep}, $A(G)$ has a bounded  approximate identity, say $(\phi_\alpha)_\alpha.$ Let $u_\alpha=\phi_\alpha\otimes\ldots\otimes\phi_\alpha.$ if $\psi_1,\psi_2,\ldots,\psi_n\in A(G),$ then $$u_\alpha(\psi_1\otimes\psi_2\otimes\ldots\otimes\psi_n)\rightarrow \psi_1\otimes\psi_2\otimes\ldots\otimes\psi_n.$$ As $A(G)\otimes A(G)\otimes\ldots\otimes A(G)$ is dense in $A^n(G),$ $u_\alpha u\rightarrow u$ $\forall\ u\in A^n(G),$ i.e., the net $(u_\alpha)_\alpha$ forms a bounded approximate identity for $A^n(G).$
    \end{proof}
    We now prove the equivalence between the underlying group $G$ and operator amenability of $A^n(G)$.
    \begin{thm}
        Let $G$ be a locally compact group. Then $A^n(G)$ is the operator amenable if and only if $G$ is amenable.
    \end{thm}
    \begin{proof}
        Suppose that the group $G$ is amenable, then the algebra 
        $$A(G^n)\cong A(G) \widehat{\otimes} A(G) \widehat{\otimes}\ldots \widehat{\otimes} A(G)$$ is operator amenable. Now, consider the identity mapping 
        $$\iota: A(G)\otimes A(G)\otimes \ldots A(G) \rightarrow A(G)\otimes A(G)\otimes \ldots \otimes A(G).$$ If we equip the space on the left hand side with the operator space projective tensor norm and the space on the right hand side by the extended Haagerup norm, then the map $\iota$ is completely contractive and hence extends to a complete contraction from $A(G^n)$ to $A^n(G).$ Also, note that the image is dense in $A^n(G).$ Thus by \cite[Theorem 7.2]{LNR}, $A^n(G)$ is operator amenable.

        Conversely, suppose that $A^n(G)$ is operator amenable. By \cite[Proposition 4.2]{TT}, the space $A^n(G)$ is a completely contractive Banach algebra and hence by \cite[Proposition 2.3]{Rua}, $A^n(G)$ has a bounded approximate identity. Thus by Theorem \ref{ABAI}, the conclusion follows.
    \end{proof}
    Here is the final theorem of this section. It gives the condition for the amenability of $A^n(G)$ in terms of the open subgroup of $G$.    
    \begin{thm}
        Let $G$ be a locally compact group and $H$ be an open subgroup of $G$ with finite index. Then $A^n(G)$ is amenable iff $A^n(H)$ is amenable.
    \end{thm}
\begin{proof}
    Let $A^n(G)$ be amenable. Let $E$ denote the natural extension map from $A(H)$ to $A(G)$ which is a complete isometry. Observe that the map $E^n= E\otimes E \otimes \ldots \otimes E: A^n(H) \rightarrow A^n(G)$ is again a complete isometry. Let $R$ denote the natural restriction map from $A(G)$ to $A(H)$ which is a complete contraction. Also, see that $R^n = R\otimes R \otimes \ldots \otimes R: A^n(G) \rightarrow A^n(H)$ is a complete contraction. As $E^n \circ R^n = \text{id}$, it shows that $R^n$ is an onto homomorphism and $A^n(H)= A^n(G)/\text{Ker} R^n$. Now, amenability of $A^n(H)$ follows from \cite[Corollary 2.3.2]{Run}.
    
    Conversely, let $A^n(H)$ be amenable. As it is given that $H$ is an open subgroup of finite index, then $H^n$ is also an open subgroup of finite index of $G^n$.
    Now, choose the  $e,x_1,x_2,\ldots,x_m$ for the left cosets of $H^n$ in $G^n$. Define a map $$\Psi: \underset{i=1}{\overset{m}{\oplus}} A^n(H) \rightarrow A^n(G)$$ as $$\Psi\left(\underset{i=1}{\overset{m}{\oplus}} u_i\right) = \underset{i=1}{\overset{m}{\oplus}} L_{x_i} u_i.$$ The functions $L_{x_i} u_i$ are supported on disjoint open sets of $G^n$, so the map $\Psi$ is a continuous homomorphism which is clearly surjective. Now, the amenability of $A^n(G)$ follows from the fact that $A^n(G)$ is a quotient of $A^n(H)$.
\end{proof}
Here are some of the corollaries of the above theorems.
\begin{cor}
    Suppose $G$ has an abelian subgroup of finite index, then $A^n(G)$ is amenable.
\end{cor}
\begin{proof}
Let $H$ be the closed abelian subgroup of $G$ having finite index , then by \cite[Theorem 4.5.5]{KL2}, $L^1(\hat{H})$ will be amenable. And we already know that for the abelian $H$, we have $$A^n(H) = \otimes_{n}^{eh}L^1(\hat{H}).$$ Since, the embedding of Max $L^1(\hat{H} \times \hat{H} \times \ldots \times \hat{H})$ inside $\otimes_{n}^{eh}L^1(\hat{H})$
is a complete contraction and $L^1(\hat{H} \times \hat{H} \times \ldots \times \hat{H})$ is amenable, it implies that $\otimes_{n}^{eh}L^1(\hat{H})$ is amenable. Therefore, the amenability of $A^n(G)$ will be obtained using the previous theorem since a closed subgroup with a finite index is open.
\end{proof}
Here is the statement that describes when is a weak$^*$ space, a complemented set in terms of the amenable group.
    \begin{cor}
        Let $G$ be a locally compact group and let $X$ be a weak*-closed $A^n(G)$-submodule of $VN^n(G).$ Then $G$ is amenable if and only if the following statements are equivalent:
        \begin{enumerate}[a)]
            \item The space $X$ is invariantly complemented
            \item The space $^\perp X$ has a bounded approximate identity.
        \end{enumerate} 
    \end{cor}
    \begin{proof}
        The forward part is a consequence of Theorem \ref{ABAI} and \cite[Proposition 6.4]{For1}, while the backward part again follows from Theorem \ref{ABAI}.
    \end{proof}
    Our next and final Corollary describes, when is the inclusion of $A^n(G)$ into $MA^n(G)$ an isometry.
    \begin{cor}
        Let $G$ be an amenable group. Then the inclusion of $A^n(G)$ into $MA^n(G)$ is an isometry.
    \end{cor}
    \begin{proof}
        For any locally compact group $G$ is amenable, it is clear that including $A^n(G)$ into $MA^n(G)$ is a contraction. Therefore, $\|u\|_{MA^n(G)}\leq \|u\|_{A^n(G)}$ for all $u\in A^n(G).$ On the other hand, using the assumption that $G$ is amenable, by Theorem \ref{ABAI} $A^n(G)$ has a bounded approximate identity, say $\{u_\alpha\}$ such that $\|u_\alpha\|\leq 1$ for all $\alpha.$ Thus, if $u\in A^n(G),$ then $$\|u_\alpha u\|_{A^n(G)}\leq \|u_\alpha\|_{A^n(G)} \|u\|_{MA^n(G)}\leq \|u\|_{MA^n(G)}.$$ Hence the proof.
    \end{proof}
    
\section{Approximation properties}
In this section, we will analyze the approximation properties (AP) of $A^n(G)$ in $M A^n(G)$ with the help of approximation identities of $A^n(G)$. First, we prove that $M A^n(G)$ is a dual Banach Space.
\begin{prop}
 Consider $G$ to be locally compact group and $Q^n(G)$ is the completion of $L^{1}(G^n)$ with respect to the norm
$$
\|f\|_{Q^n(G)}=\sup \left\{\left|\int_{G^n} f(x_1,\ldots,x_n) u(x_1,\ldots,x_n) d x_1 \ldots dx_n \right|: u \in M(A^n(G)),\|u\|_{M(A^n(G))} \leq 1\right\}
$$
then $Q^n(G)^{*}=M(A^n(G))$. In other words, every bounded linear functional $\Psi$ on $Q^n(G)$ is of the form
$$
\Psi(f)=\int_{G^n} f(x_1,\ldots,x_n) u(x_1,\ldots,x_n) d x_1 \ldots dx_n, \quad f \in L^{1}(G^n)
$$
for some $u \in M(A^n(G))$, and $\|\Psi\|=\|u\|_{M(A^n(G))}$.   
\end{prop}
   \begin{proof} 
   Observe that for every $u \in M(A^n(G))$, the map
$$
f \rightarrow \int_{G^n} f(x_1,\ldots,x_n) u(x_1,\ldots,x_n)dx_1 \ldots dx_n , \quad f \in L^{1}(G^n)
$$
extends to a bounded linear functional $\phi_{u}$ on $Q^n(G)$ and also  $\left\|\Psi_{u}\right\| \leq\|u\|_{M(A^n(G))}$. 
Conversely, let $\Psi \in Q^n(G)^{*}$ with $\|\Psi\|=1$. As the restriction of $\Psi$ to $L^{1}(G^n)$ is also a bounded linear functional on $L^{1}(G^n)$, then there exists $v \in L^{\infty}(G^n)$ with
$$
\Psi(f)=\int_{G^n} f(x_1,\ldots,x_n) v(x_1,\ldots,x_n)d x_1 \ldots dx_n, \quad f \in L^{1}(G^n)
$$
Now, as  $\left|\int_{G^n} f(x_1,\ldots,x_n) v(x_1,\ldots,x_n)dx_1 \ldots dx_n \right| \leq 1$ for all $f \in L^{1}(G^n)$ with $\|f\|_{Q^n(G)} \leq 1$, it will follow from the Hahn Banach theorem and the definition of $\|\cdot\|_{E}$ that $v$ belongs to the $\sigma\left(L^{\infty}, L^{1}\right)$ closure of the unit ball of $M(A^n(G))$. Therefore, by n-dimensional version of \cite[Lemma 5.1.4]{KL2}, $v$ is locally almost everywhere equal to some $u \in M(A^n(G))$, and $\|u\|_{M(A^n(G))} \leq 1$. This finishes the proof.
\end{proof}

The above theorem implies that $M(A^n(G))$ is a dual Banach space.
and $Q^n(G)$ denotes the predual of $M A^n(G).$
We say that the multidimensional Fourier algebra $A^n(G)$ has approximation properties (AP) in the multiplier algebra $M A^n(G)$ if there is a net $\{\phi_\alpha\}$ in $A^n(G)$ such that $\phi_\alpha \rightarrow 1$ in the $\sigma(M A^n(G),Q^n(G))$-topology, where $Q^n(G)$ denotes the predual of $M^{cb} A^n(G).$ The major result of this article is Theorem 4.2 which characterizes the AP of $A^n(G)$. 

 For  $v \in A^n(G)\cap C_c(G^n)$ and $u \in M A^n(G)$, the convolution $v \ast u \in M A^n(G)$. Further $$\|v \ast u\|_{M A^n(G)} \leq \|u\|_{M A^n(G)} \underset{G^n}{\int} |v(x_1,\ldots,x_n)|dx_1 \ldots dx_n .$$ In fact, for any $f \in L^1(G^n)$ with $\|f\|_{Q^n(G)} \leq 1$, we have
\begin{align*}
&\left|\langle v\ast u,f\rangle\right|\\ = & \left|\underset{G^n}{\int} \underset{G^n}{\int} v(y_1,\ldots,y_n)u(y_{1}^{-1}x_1,\ldots,y_{n}^{-1}x_n)f(x_1,\ldots,x_n)dy_1\ldots dy_n dx_1 \ldots dx_n\right| \\ = & \left| \underset{G^n}{\int}v(y_1,\ldots,y_n)\langle L_{(y_1,\ldots,y_n)} u,f \rangle dy_1\ldots dy_n\right|\\ \leq & \underset{G^n}{\int}|v(y_1,\ldots,y_n)| \|L_{(y_1,\ldots,y_n)} u\|_{MA^n(G)} dy_1\ldots dy_n\\ = & \|u\|_{MA^n(G)} \underset{G^n}{\int}|v(y_1,\ldots,y_n)| dy_1 \ldots dy_n.
\end{align*}
Thus for any $T \in VN^n(G), u \in A^n(G) \text{ and } v\in A^n(G)\cap C_c(G^n)$, the function
$$ \omega^n_{T,u,v} (w) = \langle T, (v \ast w)u \rangle,\ w \in MA^n(G) $$
defines a bounded linear functional on $MA^n(G)$ satisfying 
$$\|\omega^n_{T,u,v}\| \leq \|T\|\|u\| \underset{G^n}{\int}|v(y_1,\ldots,y_n)| dy_1\ldots dy_n.$$
\begin{lem}\label{Lemma}
    Let $G$ be a locally compact group. Let $T \in VN^n(G), u \in A^n(G)$ and 
    $v \in A^n(G)\cap C_c(G^n)$ with $v(x_1,\ldots,x_n) \geq 0\ \forall\ (x_1,\ldots,x_n) \in G^n$ and\\ $\underset{G^n}{\int} v(x_1,\ldots,x_n) dx_1 \ldots dx_n =1,$ then $\omega^n_{T,u,v} \in Q^n(G) \text{ and } \|\omega^n_{T,u,v}\| \leq \|T\| \|u\|.$
\end{lem}
\begin{proof}
    Without loss of generality, we assume that $u \in A^n(G) \cap C_c(G^n)$ since $Q^n(G)$ is norm closed in $M A^n(G)^*$ and $$\|\omega^n_{T,u,v}\| \leq \|T\|\|u\| \underset{G^n}{\int}|v(x_1,\ldots, x_n)| dx_1 \ldots dx_n.$$ 
    To conclude the proof, we will show that there exists $g \in L^1(G^n)$ such that 
    $$ \omega^n_{T,u,v} (u) = \underset{G^n}{\int} u(x_1,\ldots, x_n) g(x_1,\ldots, x_n)dx_1 \ldots dx_n\ \forall\ u \in MA^n(G).$$ 
    Let $K$ denotes the compact set $(\text{supp } v)^{-1} \text{supp } u$ and $w$ be any function in $MA^n(G).$ It is easy to verify that 
    $$[(v \ast w) u](x_1,\ldots ,x_n) =  [(v \ast \chi_K w) u](x_1,\ldots ,x_n)\ \forall\ (x_1,\ldots ,x_n) \in G^n.$$
    Also, we have 
    \begin{align*}
        & [(v \ast \chi_K w) u](x_1,\ldots ,x_n)\\ = & \underset{G^n}{\int}v(x_1y_1,\ldots, x_ny_n)(\chi_K w)(y^{-1}_1,\ldots,y^{-1}_n)u(x_1,\ldots ,x_n)dy_1\ldots dy_n\\ = & \underset{G^n}{\int}R_{(y_1,\ldots,y_n)}v(x_1,\ldots ,x_n)u(x_1,\ldots ,x_n)(\chi_K w)(y^{-1}_1,,\ldots,y^{-1}_n)dy_1\ldots dy_n.
    \end{align*}
    Now, define a map $\Psi: G^n \rightarrow A^n(G)$
    $$ \Psi(y_1,\ldots, y_n)(x_1,\ldots ,x_n)= R_{(y_1,\ldots,y_n)}v(x_1,\ldots ,x_n)u(x_1,\ldots ,x_n),$$
     for all $(y_1,\ldots, y_n)\in G^n$. Observe that the map $\Psi$ is norm continuous because the right translation of $v$ is a norm continuous map from $G^n \text{ to } A^n(G).$ Further, it is clear that $\Psi(y_1,\ldots, y_n)=0$ if $(y_1,\ldots, y_n) \notin K.$ Therefore the map $\Psi$ is bounded by the compactness of $K$.\\
     Consider a measure $\nu$ on $G^n$ defined as
     $$ d\nu(y_1,\ldots, y_n)=(\chi_K w)(y^{-1}_1,,\ldots,y^{-1}_n)dy_1 \ldots dy_n.$$
     The measure $d\nu$ is a bounded radon measure because $u$ is bounded and $K$ is compact. So, there exists a function $\zeta \in A^n(G) $ such that 
     $$ \zeta = \underset{G^n}{\int} \Psi(y_1,\ldots, y_n)d\nu(y_1,\ldots, y_n) $$
     and
     $$ \langle T, \zeta \rangle = \underset{G^n}{\int} \langle T, \Psi(y_1,\ldots, y_n) \rangle d\nu(y_1,\ldots, y_n)\ \forall\ T \in VN^n(G).$$
     Let $T = \delta_{x_1} \otimes \ldots \delta_{x_n}$ for some $(x_1,\ldots ,x_n) \in G^n$. Then 
     \begin{align*}
         & \zeta(x_1,\ldots ,x_n)\\
         = & \underset{G^n}{\int} \Psi(y_1,\ldots, y_n)(x_1,\ldots ,x_n) d\nu(y_1,\ldots, y_n)\\
         = & \underset{G^n}{\int}R_{(y_1,\ldots,y_n)}v(x_1,\ldots ,x_n)u(x_1,\ldots ,x_n)(\chi_K w)(y^{-1}_1,\ldots,y^{-1}_n)dy_1\ldots dy_n\\
         = & [(v \ast w) u](x_1,\ldots ,x_n).
     \end{align*}
     Therefore, $\zeta= [(v \ast w) u]$ and
     \begin{align*}
         & \omega^n_{T,u,v} (w) = \langle T, [(v \ast w) u]  \rangle = \langle T, \zeta \rangle\\
         = & \underset{G^n}{\int} \langle T, \Psi(y_1,\ldots, y_n) \rangle d\nu(y_1,\ldots, y_n)\\
         = & \underset{G^n}{\int} (\chi_K w)(y^{-1}_{1},\ldots,y^{-1}_{n})\langle T, \Psi(y_{1},\ldots, y_{n}) \rangle dy_{1} \ldots dy_{n}\\ 
         = & \underset{G^n}{\int}  w(y^{-1}_{1},\ldots,y^{-1}_{n})\langle T, \Psi(y_{1},\ldots, y_{n}) \rangle \chi_K(y^{-1}_{1},\ldots,y^{-1}_{n}) dy_{1} \ldots dy_{n}\\
         = & \underset{G^n}{\int}  w(y_{1},\ldots,y_{n})\langle T, \Psi(y^{-1}_{1},\ldots, y^{-1}_{n}) \rangle \chi_K(y_{1},\ldots,y_{n})\Delta y^{-1}_1 \ldots \Delta y^{-1}_n dy_{1}\ldots dy_{n}\\
         = & \underset{G^n}{\int}  w(y_{1},\ldots,y_{n})g(y_{1},\ldots,y_{n})dy_{1} \ldots dy_{n},
    \end{align*}
    where $$ g(y_{1},\ldots,y_{n})= \langle T, \Psi(y^{-1}_{1},\ldots, y^{-1}_{n}) \rangle \chi_K(y_{1},\ldots,y_{n})\Delta y^{-1}_1 \ldots \Delta y^{-1}_n.$$ Since,
    \begin{align*}
         & \underset{G^n}{\int} |g(y_{1},\ldots,y_{n})|dy_{1} \ldots dy_{n}\\
         = & \underset{G^n}{\int} |\langle T, \Psi(y^{-1}_{1},\ldots, y^{-1}_{n}) \rangle| \chi_K(y_{1},\ldots,y_{n})\Delta y^{-1}_1 \ldots \Delta y^{-1}_n dy_{1} \ldots dy_{n}\\
         \leq & \|T\| \underset{G^n}{\int}  \| \Psi(y^{-1}_{1},\ldots, y^{-1}_{n}) \| \chi_K(y_{1},\ldots,y_{n})\Delta y^{-1}_1 \ldots \Delta y^{-1}_n dy_{1} \ldots dy_{n}\\
         \leq & M \|T\| \lambda(K^{-1}) < \infty,
    \end{align*}
    where $M$ is the upper bound of $\Psi$ and $g \in L^1(G^n)$ as required.    
\end{proof}
The following statement is the main theorem of our paper, which characterizes the AP of $A^n(G)$ and serves as an analogous representation of the weak containment property. We must mention here that the following is the analogue of \cite[Theorem 3.2]{Miao2}.
\begin{thm}\label{main thm}
    Let $G$ be a locally compact group, then the following statements are equivalent.
    \begin{enumerate}[(i)]
        \item $A^n(G)$ has the AP in the corresponding multiplier algebra $M A^n(G)$.
        \item For any sequence $\{v_n\} \in A^n(G)$ and any sequence $\{f_n\} \in \otimes^n_h C_r^*(G)$ with $$\underset{n=1}{\overset{\infty}{\sum}}\|v_n\| \|f_n\| < \infty \text{ and } \langle u, \underset{n=1}{\overset{\infty}{\sum}}v_n f_n \rangle=0\ \forall\ u \in A^n(G)$$ will imply $$\underset{n=1}{\overset{\infty}{\sum}} \langle v_n, f_n \rangle =0.
        $$
        \item For any sequence $\{v_n\} \in A^n(G)$ and any sequence $\{T_n\} \in VN^n(G)$ with
        $$\underset{n=1}{\overset{\infty}{\sum}}\|v_n\| \|T_n\| < \infty \text{ and } \langle u, \underset{n=1}{\overset{\infty}{\sum}}v_n T_n \rangle=0\ \forall\ u \in A^n(G)$$ will imply $$\underset{n=1}{\overset{\infty}{\sum}} \langle v_n, T_n \rangle =0.
        $$
        \item There is a net  $\{u_\alpha\} \in A^n(G)$ such that for any sequence $\{v_n\} \in A^n(G)$ with $\|v_n\|_{A^n(G)} \rightarrow 0$, we have
        $$\underset{\alpha}{\lim}\|u_\alpha u- u\|_{A^n(G)} \rightarrow 0 \text{ uniformly for } u \text{ in } \{v_n\}.$$
        \item There is a net  $\{u_\alpha\} \in A^n(G)$ such that
        $$\underset{\alpha}{\lim}\|u_\alpha u- u \|_{A^n(G)} \rightarrow 0 \text{ uniformly for } u \text{ in any compact set of } A^n(G).$$
        \item There is a net $\{u_\alpha\} \in A^n(G)$  such that $ \langle u_\alpha, f \rangle \rightarrow \langle 1,f \rangle$	 uniformly for $f$ in any compact subset of $Q^n(G).$
    \end{enumerate}
\end{thm}
\begin{proof}
     $(i) \implies (iv)$ Consider $X=(Q^n(G)\oplus Q^n(G) \oplus \ldots \oplus Q^n(G))_0$ i.e., the space of all sequences $\{u_n\} \in A^n(G)$ such that $\|u_n\|_{A^n(G)} \rightarrow 0$ and have sup-norm i.e.,$\|(u_n)\|=\text{sup}_n \|u_n\|_{A^n(G)}$. Observe that if $(v_n) \in X$ then $(u v_n -v_n) \in X$ for any $u \in A^n(G).$\\
     Now, define
     $$\Psi: A^n(G) \rightarrow X $$
     $$\Psi(u)= u v_n-v_n$$
     We claim that $0$ is in the closure of $\Psi(A^n(G))$ in the weak topology of $X$. Using hypothesis (i), there exists a net $\{u_\alpha\} \in A^n(G)$ with $u_\alpha \rightarrow 1$ in the associated weak$^*$ topology i.e., $\sigma(M A^n(G), Q^n(G))$-topology.
     
     Now, choose an element $w \in A^n(G) \cap C_c(G^n)$ with the property that $w(x_1,\ldots,x_n) \geq 0\ \forall\ (x_1,\ldots,x_n)\in G^n$ and $\underset{G^n}{\int} w(x_1,\ldots,x_n) dx_1 \ldots dx_n = 1.$ For all $\alpha$, the convolution $w \ast u_\alpha \in A^n(G)$ due to the compact support of $w.$
     We first show that $\Psi(u \ast u_\alpha) \rightarrow 0$ weakly in $X$. Let $T \in X^*,$ i.e., there is a sequence $T_n \in VN^n(G)$ with $\|T\|= \underset{n=1}{\overset{\infty}{\sum}} \|T_n\| < \infty$ and $$ \langle T,u_n \rangle = \underset{n=1}{\overset{\infty}{\sum}} \langle T_n, u_n \rangle \text{ for any } (u_n) \in X.$$
     Then 
     \begin{align*}
         \langle T, \Psi(w \ast u_\alpha) \rangle & = \underset{n=1}{\overset{\infty}{\sum}} \langle T_n, (w\ast u_\alpha) v_n-v_n \rangle\\
         &= \underset{n=1}{\overset{\infty}{\sum}} \langle T_n, (w\ast u_\alpha) v_n \rangle- \underset{n=1}{\overset{\infty}{\sum}}\langle T_n, v_n \rangle\\
         &= \underset{n=1}{\overset{\infty}{\sum}} \langle  \omega^n_{T_n,v_n,w},u_\alpha  \rangle-\underset{n=1}{\overset{\infty}{\sum}}\langle T_n, v_n \rangle = \left\langle u_\alpha,\underset{n=1}{\overset{\infty}{\sum}}   \omega^n_{T_n,v_n,w}  \right\rangle-\underset{n=1}{\overset{\infty}{\sum}}\langle T_n, v_n \rangle.
     \end{align*}
     From Lemma \ref{Lemma}, $\omega^n_{T_n,v_n,w} \in Q^n(G)$ and $$ \underset{n=1}{\overset{\infty}{\sum}} \|\omega^n_{T_n,v_n,w}\|_{Q^n(G)} \leq \underset{n=1}{\overset{\infty}{\sum}}\|T_n\|_{VN^n(G)} \|v_n\|_{A^n(G)} < \infty.$$
     This shows that $ \underset{n=1}{\overset{\infty}{\sum}} \omega^n_{T_n,v_n,w} \in Q^n(G).$ So,
     $$\langle u_\alpha,\underset{n=1}{\overset{\infty}{\sum}}   \omega^n_{T_n,v_n,w}  \rangle \rightarrow \langle 1,\underset{n=1}{\overset{\infty}{\sum}}   \omega^n_{T_n,v_n,w}  \rangle= \underset{n=1}{\overset{\infty}{\sum}} \langle T_n, (w\ast 1) v_n \rangle= \underset{n=1}{\overset{\infty}{\sum}} \langle T_n, v_n \rangle$$
     Hence, $\langle T, \Psi(w \ast u_\alpha) \rangle \rightarrow 0.$ We know that $\Psi(A^n(G))$ is convex and $0$ is in norm closure of $\Psi(A^n(G))$, so for every positive integer there is $u_{(i,A)} \in A^n(G),$ where $A= \text{ countable subset of }\{v_n\}$ such that $$ \| u_{(i,A)} v_n - v_n \|_{A^n(G)} \leq \|\Psi(u_{(i,A)})\|_X < \frac{1}{i} $$ for every $n$. Observe that the $(u_{(i,A)})$ is a net directed by $(i_1,A_1) \leq (i_2,A_2)$ if $i_1 \leq i_2$ and $A_1 \subseteq A_2$ which satisfies the requirement.\\
     $(iv) \implies (iii)$ Let $u_\alpha$ be a net which satisfies the condition (iv). Consider the sequence $\{T_n\} \in VN^n(G)$  and $\{v_n\}$ with the property that $$\underset{n=1}{\overset{\infty}{\sum}} \|T_n\|_{VN^n(G)}\|v_n\|_{A^n(G)} < \infty.$$ Without the loss of generality, let us assume that $\| v_n\|_{A^n(G)}=1\ \forall\ n.$ Then $\underset{n=1}{\overset{\infty}{\sum}} \|T_n\|_{VN^n(G)} < \infty,$ the series convergence implies that there are positive integers $i_1 < i_2 <\ldots $ such that for each k, $\underset{n=i_k}{\overset{i_{k+1}-1}{\sum}}\|T_n\| < \frac{1}{2^k}.$
     Now let us denote,
     $$ A= \{v_1,v_2,\ldots, v_{i_1-1}\} \cup \left\{ \frac{v_n}{k}:k=1,2,3 \ldots, i_k < n< i_{k+1} \right\}.$$
     The existence of obvious order makes it a sequence converging to zero in the norm. Using the hypothesis of (iv), for any $\epsilon >0$, there exists  $\alpha_0$ with the property that $\|u_\alpha w-w \| < \epsilon$ and $\underset{n=1}{\overset{i_{1}-1}{\sum}} \|u_\alpha v_n-v_n \| < \epsilon\ \forall\ w \in A\ \text{and}\ \alpha \geq \alpha_0.$ Now for all $\alpha \geq \alpha_0,$
     \begin{align*}
          & |\langle u_\alpha,\underset{n=1}{\overset{\infty}{\sum}} v_nT_n \rangle  - \underset{n=1}{\overset{\infty}{\sum}} \langle T_n, v_n \rangle| =  |\underset{n=1}{\overset{\infty}{\sum}} \langle T_n, u_\alpha v_n- v_n \rangle \leq  \underset{n=1}{\overset{\infty}{\sum}} \|T_n\| \|u_\alpha v_n- v_n\|\\
         = & \underset{n=1}{\overset{i_1-1}{\sum}} \|T_n\| \|u_\alpha v_n- v_n\| + \underset{k=1}{\overset{\infty}{\sum}} \underset{n=i_k}{\overset{i_{k+1}-1}{\sum}}\|T_n\| \|u_\alpha v_n- v_n\|\\
         \leq & (\underset{n=1}{\overset{\infty}{\sum}}\|T_n\|)\underset{n=1}{\overset{i_1-1}{\sum}}\|u_\alpha v_n- v_n\| + \underset{k=1}{\overset{\infty}{\sum}} \underset{n=i_k}{\overset{i_{k+1}-1}{\sum}}\|T_n\| \|u_\alpha \frac{v_n}{k}- \frac{v_n}{k}\|k\\
         < & \epsilon(\underset{n=1}{\overset{\infty}{\sum}}\|T_n\|)+\underset{k=1}{\overset{\infty}{\sum}} \underset{n=i_k}{\overset{i_{k+1}-1}{\sum}}\|T_n\|\epsilon k  \leq  \epsilon(\underset{n=1}{\overset{\infty}{\sum}}\|T_n\|)+\epsilon(\underset{k=1}{\overset{\infty}{\sum}}\frac{k}{2^k}).
     \end{align*}
     From the assumption of $(iii)$, $\langle u_\alpha,\underset{n=1}{\overset{\infty}{\sum}} v_nT_n \rangle=0\ \forall\ \alpha$ and $\epsilon$ is arbitrary, hence $\underset{n=1}{\overset{\infty}{\sum}} \langle T_n, v_n \rangle=0$, as required.\\
     $(iii) \implies (ii)$  is trivial.\\
     $(ii) \implies (i)$ In contradiction, assume that $1 \notin $ closure of $A^n(G)$ in $\sigma(M A^n(G),Q^n(G))$-topology. Then using the Hahn- Banach Theorem, there is a $f \in Q^n(G)$ such that $\langle1,f \rangle= 1 \text{ and } \langle u, f \rangle =0\ \forall\ u \in A^n(G)$.
     Using the characterization of $Q^n(G)$, (same as in \cite{Miao}) there exists a sequence $(v_n) \in A^n(G)$ and $(f_n) \in \otimes^n_h C_r^*(G)$ with $\|v_n\|_{A^n(G)} \|f_n\|_{\otimes^n_h C_r^*(G)} < \infty$ and $f= \underset{n=1}{\overset{\infty}{\sum}} v_n f_n$.  So, we have $\langle 1, \underset{n=1}{\overset{\infty}{\sum}} v_n f_n \rangle = \underset{n=1}{\overset{\infty}{\sum}} \langle v_n, f_n \rangle,$ where $\langle u, \underset{n=1}{\overset{\infty}{\sum}} v_n f_n \rangle = 0\ \forall\ u \in A^n(G).$ Hence, $(ii)$ fails.\\
     $(v) \implies (iv)$ is trivial.\\
     $(iv) \implies (v)$ Let $(u_\alpha)$ be the net satisfying the condition $(iv)$, then $(u_\alpha)$ will also satisfy the condition of $(v)$. If $S$ is a compact subset of $A^n(G)$, then by \cite[Proposition 1.e.2]{LT} there exists a sequence $\{v_n\} \in A^n(G)$ with $\|v_n\| \rightarrow 0$ with $S \subseteq \overline{\text{conv } \{v_n\}}.$ By $(iv)$, for every $\epsilon > 0,$ there exists a $\alpha_0$ such that $\| u_\alpha v_n-v_n\| < \epsilon\ \forall\ \alpha \geq \alpha_0.$ Also, $\|u_\alpha v-v \| < \epsilon\ \forall\ n $  $\forall v \in \text{conv}\{v_n\}.$ Now, for any $w \in \overline{\text{conv}\{v_n\}}$, there is a sequence $\{w_i\} \in \text{conv}\{v_n\}$ such that $\{w_i\} \rightarrow w$ in norm. Hence, $\|u_\alpha w - w \| = \lim_{i} \|u_\alpha w_i - w_i \|  \leq \epsilon\  \forall\ \alpha \geq \alpha_0,$ as required.\\
     $(vi) \implies (i)$ is trivial.\\
     $(i) \implies (vi)$ Let $u_\alpha \rightarrow 1$ in $\sigma(M A^n(G),Q^n(G)) \text{-topology}.$ For any compact set $K \subseteq Q^n(G)$, there exists a sequence $\{f_n\} \in Q^n(G)$ such that $K \subseteq \overline{\text{conv}\{f_n\}}$. Now it is enough to prove that $\langle u_\alpha,f \rangle \rightarrow \langle 1, f \rangle\ \forall\ f \in \{f_n\}.$ Consider $X=(Q^n(G)\oplus Q^n(G) \oplus \ldots \oplus Q^n(G))_0$, then the net $\Phi_\alpha =((u_\alpha-1)f_n)_{n=1}^{\infty} \in X$ because $\|f_n\| \rightarrow 0.$ Now, we claim that $\Phi_\alpha\rightarrow 0$ weakly in $X.$ Let $F \in X^*$, then there exists a sequence $\{x_i\}$ in $Q^n(G)^*= M A^n(G)$ with $\|F\| =\underset{i=1}{\overset{\infty}{\sum}} \|x_i\|$ and $\langle F, (g_i) \rangle = \underset{i=1}{\overset{\infty}{\sum}} \langle x_i, g_i \rangle$ for any $(g_i) \in X.$ Now $x_if_i \in (MA^n(G))^*$ with $\langle x_i f_i,x \rangle= \langle f_i,x_ix \rangle$ for $x \in MA^n(G)$. As $Q^n(G)$ is norm closure of $L^1(G^n)$ in $(MA^n(G))^*,$ then $x_if_i \in Q^n(G)$ and also $\underset{i=1}{\overset{\infty}{\sum}}x_if_i \in Q^n(G)$. So we have, 
     $$\langle F,\Phi_\alpha \rangle = \underset{i=1}{\overset{\infty}{\sum}} \langle x_i, (u_\alpha-1)f_i \rangle \underset{i=1}{\overset{\infty}{\sum}}  \langle (u_\alpha-1), x_if_i \rangle =  \left\langle u_\alpha-1, \underset{i=1}{\overset{\infty}{\sum}} x_if_i \right\rangle \rightarrow 0,$$ as $u_\alpha \rightarrow 1$ in weak$^*$- topology.
         For each $k,$ there is a convex combination of $\{\Phi_\alpha\}$i.e., $\underset{i=1}{\overset{m_k}{\sum}}\beta_i \Phi_{\alpha_i}$ with the property $\|\underset{i=1}{\overset{m_k}{\sum}}\beta_i \Phi_{\alpha_i}\| < \frac{1}{k}$. Let us denote $v_{(K,k)} = \underset{i=1}{\overset{m_k}{\sum}}\beta_i u_{\alpha_i}$, then for each $n$,
         \begin{align*}
             |\langle v_{(K,k)},f_n \rangle-\langle 1,f_n \rangle| =& |\langle v_{(K,k)}-1,f_n \rangle| = |\langle 1, (v_{(K,k)}-1)f_n \rangle| \\ \leq & \|(v_{(K,k)}-1)f_n\|
              =\|(\underset{i=1}{\overset{m_k}{\sum}}\beta_i u_{\alpha_i}-1)f_n\|\\
              = & \|(\underset{i=1}{\overset{m_k}{\sum}}\beta_i\Phi_{\alpha_i}\| < \frac{1}{k}.
         \end{align*}
         Hence, the net $v_{(K,k)}$ satisfies the requirement.
\end{proof}
Here is the result related to the AP of subgroups.
\begin{cor}
    Let $H$ be a closed subgroup of locally compact group $G$. If $A^n(G)$ has AP in $MA^n(G)$, then  $A^n(H)$ has AP in $MA^n(H).$
\end{cor}
\begin{proof}
    As $A^n(G)$ has AP in $MA^n(G)$ then using Theorem \ref{main thm}, there exists a net in $\{u_\alpha\} \in A^n(G)$ such that $\|u_\alpha u - u \|_{A^n(G)} \rightarrow 0$ uniformly for $u$ in any compact subset of $A^n(G).$ Let $\{v_n\}$ be any sequence in $A^n(H)$ with $\|v_n\|_{A^n(H)} \rightarrow 0.$ Now, there exists a sequence $\{\tilde{v_n}\} \in A^n(G)$ with $\|\tilde{v_n}\| \rightarrow 0$ and $\tilde{v_n}(h) = v_n(h)\ \forall\ h \in H.$ Then, $\|u_\alpha \tilde{v_n} - \tilde{v_n} \|_{A^n(G)} \rightarrow 0\ \forall\ n.$ Since the restriction map from $A^n(G)$ to $A^n(H)$ is a contraction, then $\|(u_\alpha)_|{H} v_n - v_n \|_{A^n(H)} \rightarrow 0.$ Hence the net $(u_\alpha)_|{H} \in A^n(H)$ satisfies $(iv)$ of Theorem \ref{main thm} and $A^n(H)$ has AP in $MA^n(H).$
\end{proof}

\section{Weak amenability}
This section aims to prove Theorem 5.2. Let us begin with a proposition. 
\begin{prop}\label{Mp}
    Let $G$ be a locally compact group then
    \begin{enumerate}[(i)]
        \item $M^{cb}A(G) \odot \ldots \odot M^{cb}A(G) \subseteq M^{cb}A^n(G)$
        \item $A^n(G) \subseteq \overline{M^{cb}A(G) \odot \ldots\odot M^{cb}A(G)}^{\|.\|_{cbm}}.$
    \end{enumerate}
\end{prop}
\begin{proof}
    (i) Let $\phi_1,...,\phi_n\in M^{cb}A(G)$. Define a map $$m_{\phi_i}: A(G) \rightarrow A(G)$$ $$m_{\phi_i}(f) = \phi_i f\ \forall\ f \in A(G).$$ Observe that the map $m_{\phi_i}$ is completely bounded for all $i \in \{1,2,...n\}$ and the corresponding map $$m_{\phi_{1}} \otimes id \otimes \ldots \otimes id :A^n(G) \rightarrow A^n(G),$$ $$(m_{\phi_1} \otimes id \otimes \ldots \otimes id)(u) = (\phi_1 \otimes 1 \otimes \ldots \otimes 1)u$$ is also completely bounded. Hence, $\phi_1 \otimes 1 \otimes \ldots \otimes 1 \in M^{cb}A^n(G).$ Similarly, $1 \otimes \phi_2 \otimes \ldots \otimes 1, 1 \otimes 1 \otimes \phi_3 \otimes \ldots \otimes 1, ..., 1 \otimes 1 \otimes \ldots \otimes \phi_n \in M^{cb}A^n(G).$ Since $$\phi_1 \otimes \ldots \otimes \phi_n= (\phi_1 \otimes 1 \otimes \ldots \otimes 1) \cdots (1 \otimes 1 \otimes \ldots \otimes \phi_n),$$then $\phi_1 \otimes \ldots \otimes \phi_n \in M^{cb}A^n(G).$\\
    (ii)The norm inequality follows from the fact that $M^{cb}A(G)$ is a Banach algebra with $\|.\|_{cbm}$ norm.
\end{proof}

Here is the theorem for this section. This is the analogue of the \cite[Theorem 4.8]{Miao2} 
    \begin{thm}
        Let $C>0.$ Then the following statements are equivalent in a locally compact group $G.$
        \begin{enumerate}[(i)]
            \item $G$ is weakly amenable with constant $C.$
            \item There exists a net $(u_\alpha)_\alpha$ consisting of compactly supported elements of $A(G)\otimes A(G)\otimes\ldots\otimes A(G)$ such that $\|u_\alpha\|_{\mbox{cbm}}\leq C\ \forall\ \alpha$ and $u_\alpha u\rightarrow u$ in $A^n(G)$ for all $u\in A^n(G).$
            \item There exists a net $(u_\alpha)_\alpha$ consisting of compactly supported elements of $M^{\mbox{cb}}(A^n(G))$ such that $\|u_\alpha\|_{\mbox{cbm}}\leq C\ \forall\ \alpha$ and $u_\alpha u\rightarrow u$ in $A^n(G)$ $\forall\ u\in A^n(G).$
        \end{enumerate}
    \end{thm}
    \begin{proof}
        (i) $\implies$ (ii) Let $G$ be weakly amenable, then there exist $C > 0$ and a net $(\phi_\alpha)_\alpha \subseteq A(G)$ of compactly supported elements such that $\| \phi_\alpha\|_{cbm} \leq C\ \forall\ \alpha $ with $\phi_\alpha \phi \rightarrow \phi$ in $A(G)$ $\forall\ \phi \in A(G).$
        
        Consider $u_\alpha = \underbrace{ \phi_\alpha \otimes \ldots \otimes \phi_\alpha}_n$ and if $\psi_1, \psi_2,..., \psi_n \in A(G),$ then $$u_\alpha(\psi_1 \otimes \ldots \otimes \psi_n) \rightarrow_\alpha \psi_1 \otimes \ldots \otimes \psi_n$$ in $A^n(G)$. Let $u \in A^n(G)$ be arbitrary and $\epsilon > 0.$ Fix an element $u_0 \in \underbrace{A(G) \odot \ldots \odot A(G)}_n$ so that $\|u-u_0\|_{eh} \leq \epsilon/3.$ Also there exists $\alpha_0$ with the property that $\|u_\alpha u_0-u_0\|_{eh} < \epsilon/3\ \forall\ \alpha \geq \alpha_0.$ Now, for any $\alpha \geq \alpha_0,$ $$\| u_\alpha u- u\|_{eh} \leq \|u_\alpha u- u_\alpha u_0\|_{eh}+\|u_\alpha u_0- u_0\|_{eh}+\|u_0-u\|_{eh} < \epsilon.$$
        (ii) $\implies$ (iii) It follows immediately from Proposition $\ref{Mp}$ and the fact that $A(G) \subseteq M^{cb}A(G).$\\
        (iii) $\implies$ (i) Let $(v_\alpha)$ be a net of compactly supported elements of $M^{cb}(A^n(G))$ such that $$\|v_\alpha\|_{cbm} \leq C\ \forall\ \alpha \text{ and } v_\alpha v \rightarrow v  \text{ in } A^n(G).$$ By \cite{TT}, $A^n(G)$ is a regular semisimple completely contractive Banach algebra with pointwise multiplication whose Gelfand spectrum is homomorphic to $\underbrace{G \times G\times \cdots \times G}_n$, which implies $(v_\alpha) \subseteq A^n(G).$ \\
        Let $m_u: A^n(G) \rightarrow A^n(G)$ be defined as $m_u(v)= uv\ \forall\ v \in A^n(G),$ and $R_u$ be its dual map from $VN^n(G)$ to $VN^n(G)$. Also, $\|R_u\|_{cb} = \|u\|$ and $\|R_{u_\alpha}\|_{cb} \leq C\ \forall\ \alpha$. Hence, $R_{u_\alpha}(T) \rightarrow T\ \forall\ T \in VN^n(G)$ in weak$^*$ topology of $VN^n(G)$. \\
        Now, define $$\Phi: VN(G) \rightarrow VN^n(G)$$  $$\Phi(T)= T\otimes I \otimes \ldots \otimes I$$ and observe that $\Phi$ is weak$^*$-continuous as well as completely contractive. In fact it is the dual of $(\text{id}\otimes \delta_e \otimes \delta_e \otimes \ldots \delta_e) .$ One can also note that the multiplication map $m:T_1 \otimes T_2 \otimes \ldots \otimes T_n \rightarrow T_1T_2\ldots T_n$ extends uniquely to a weak$^*$-continuous completely contractive map from $VN^n(G)$ to $VN(G)$ and $$m\circ \Phi(T)= m(T \otimes I \otimes \ldots \otimes I)=T\ \forall\ T\in VN(G).$$ Thus, we have $(m\circ R_{u_\alpha} \circ \Phi )_\alpha $ as a net of weak$^*$ continuous maps on $VN(G)$ whose completely bounded norm is uniformly bounded by $C$. Also, $(m\circ R_{u_\alpha} \circ \Phi )_\alpha(\lambda_t)=u_\alpha(t,e,e,\ldots,e)\lambda_t\ \forall\ t\in G.$\\
        Let $f_\alpha: G\rightarrow \mathbb{C}$ be a function defined as $f_\alpha(t)=u_\alpha(t,e,e,\ldots,e)$ and assume that supp $u_\alpha \subseteq \underbrace{K_\alpha \times K_\alpha \times \cdots \times K_\alpha}_n$ for some compact set $K_\alpha \subseteq G$.\\ Let $g_\alpha \in A(G)$ be a function with compact support taking value $1$ on $K_\alpha$ and $e$. Then, $$\underbrace{g_\alpha \otimes \ldots \otimes g_\alpha}_n \in A^n(G) \text{ as well as } u_\alpha(\underbrace{g_\alpha \otimes \ldots \otimes g_\alpha}_n) \in A^n(G).$$ Now, it can be seen that $$f_\alpha = (\text{id} \otimes \delta_e \otimes \ldots \otimes \delta_e)(u_\alpha(g_\alpha \otimes \ldots \otimes g_\alpha)) \in A(G).$$
        Also, for any $T \in VN(G)$ and $f\in A(G)$, we have
        \begin{align*}
        \langle f_\alpha f-f,T\rangle =&\langle f, f_\alpha.T -T\rangle\\=& \langle f,(m\circ R_{u_\alpha} \circ \Phi )(T)-T\rangle\\=& \langle m_*(f), (R_{u_\alpha}-\text{id}) \circ \Phi(T)\rangle \rightarrow_\alpha 0.
        \end{align*}
        Thus, $f_{\alpha} f\rightarrow f$ in weak topology of $A(G)$. Fix $0\neq f \in A(G)$. Further, it is known that norm closure and weak closure are equal for convex sets, then from the above $\exists$ a net $(f^{'}_{\beta})_{\beta}$ in $A(G)$(depending upon $f$) such that $$\underset{\beta}{\text{sup}} \|f^{'}_{\beta}\|_{cbm} \leq C \text{ and }\|f^{'}_{\beta} f-f\|_{cbm} \leq \|f^{'}{_\beta} f-f\|_{A(G)}\rightarrow_{\beta} 0.$$ It implies that $(A(G),\|.\|_{cbm})$ has an approximate unit bounded by $C$ in the completely bounded norm. Now, the weak amenability of $G$ follows from \cite[Theorem 1]{Wic}. 
        \end{proof}


\begin{thebibliography}{aaa}
    \bibitem{ER1990} \textsc{E. G. Effros} and \textsc{Z-J. Ruan}, On approximation properties for operator spaces, {\it Internat. J.Math.} 1 (1990), 163-187. 
    \bibitem{ER1} \textsc{E. G. Effros} and \textsc{Z-J. Ruan}, On two quantum tensor products, in {\it Operator algebras, Mathematical Physics and low dimensional Topology, (Istanbul 1991) Res. Math. Notes}, 5 (1993) 125-145.
    \bibitem{ER2} \textsc{E. G. Effros} and \textsc{Z-J. Ruan}, Operator space tensor products and Hopf convolution algebras, {\it J. Op. Th.}, 50 (2003) 131-156.
    \bibitem{ER3} \textsc{E. G. Effros} and \textsc{Z-J. Ruan}, Operator spaces, Oxford University Press, 2000.
    \bibitem{Eym1} \textsc{P. Eymard}, {\it L'alg\`{e}bre de Fourier d'un groupe localement compact}, {\it Bull. Soc. Math. France}, 92 (1964), 181-236.
    \bibitem{For1} \textsc{B. Forrest}, Amenability and bounded approximate identities in ideals of A(G). Illinois J. Math. 34 (1990), no. 1, 1-25.
    \bibitem{GS} \textsc{C. C. Graham} and \textsc{B. M. Schreiber}, Bimeasure algebras on LCA groups, {\it Pacific J. Math.}, 115 (1984) 91-127.
    \bibitem{Gr} \textsc{A. Grothendieck}, Produits Tensoriels Topologiques et Espaces Nucl´eaires, {\it Mem. Amer. Math. Soc. No.}, 16 (1955).
   \bibitem{HC}\textsc{U. Haagerup} and \textsc{M.Cowling} Completely bounded multipliers of the Fourier algebra of a simple Lie group of real rank one, {\it Invent. Math.}, 96(3)(1989):507–549.
    \bibitem{HK} \textsc{U. Haagerup} and \textsc{J. Kraus},  Approximation properties for group $C^*$-algebras and group von Neumann algebras, {\it Trans. Amer. Math. Soc. 344}, No. 2 (1994), 667-699.
    \bibitem{J} \textsc {B. E. Johnson},  Cohomology in Banach algebras, {\it Mem. Amer. Math. Soc.} 127 (1972).
    \bibitem{J2} \textsc{B. E. Johnson}, Non-amenability of the Fourier algebra of a compact group, {\it J. London Math. Soc.} (2) 50 (1994), 361–374.
    \bibitem{jun} \textsc {M. Junge} and \textsc {Z. Ruan}, Approximation properties for non-commutative $L_p$-spaces associated with discrete groups, {\it Duke Math. J.}, 117(2) (2003), 313–341.
    \bibitem{K1} \textsc{E. Kaniuth}, {\it A course in commutative Banach algebras}, Springer, 2009.
    \bibitem{KL2} \textsc{E. Kaniuth} and \textsc{A. T.-M. Lau}, {\it Fourier and Fourier–Stieljes Algebras on Locally Compact Groups}, {\it Mathematical Surveys and Monograph}s, 231 (American Mathematical Society, Providence, RI, 2018).
    \bibitem{LNR} \textsc{A. Lambert}, \textsc{M. Neufang} and \textsc{V. Runde}, Operator space structure and amenability for Fig\`{a}-Talamanca-Herz algebras, {\it J. Funct. Anal.}, 211 (2004) 245-269.
    \bibitem{Lep} \textsc{H. Leptin}, Sur l'alg\`{e}bre de Fourier d'un groupe localement compact, {\it C. R. Acad. Sci. Paris Ser.} A-B, 266 (1968) A1180-A1182.
    \bibitem{LT} \textsc{J. Lindenstrauss and L. Tzafriri}, {\it Classical Banach Spaces I}, {\it Springer-Verlag, Berlin, Heidelberg}, New York, 1977.
    \bibitem{Miao} \textsc{T. Miao}, Predual of the multiplier algebra of Ap(G) and amenability,{\it Can. J. Math. 56 (2)}
(2004), 344–355.
    \bibitem{Miao2} \textsc{T. Miao}, Approximation properties and approximate identities of $A_{p}(G)$, {\it Trans. Amer. Math. Soc.}, 361(3)(2009), 1581-1595.
    \bibitem{P}\textsc{J. P. Pier}, Amenable Locally Compact Groups, Wiley, New York, 1984.
    \bibitem{Pis} \textsc{G. Pisier}, Introduction to the Theory of Operator Spaces, {\it London Mathematical Society}, 2003.
    \bibitem{Rua} \textsc{Z-J. Ruan}, The operator amenability of $A(G),$ {\it Amer. J. Math.}, 117 (1995) 1449-1474.
    \bibitem{Run} \textsc{Volker Runde}, Lectures on amenability, {\it Lecture Notes in Mathematics}, vol. 1774, Springer-Verlag, Berlin, 2002.
    \bibitem{Sp} \textsc{N. Spronk}, Operator amenability of Fourier–Stieltjes algebras, {\it Mathematical Proceedings of the Cambridge Philosophical Society} 136 (2001): 675 - 686.
    \bibitem{TT} \textsc{I. G. Todorov} and \textsc{L. Turowska}, Multipliers of multidimensional Fourier algebras, {\it Operators and matrices}, 4 (2010) 459-484.
    \bibitem{Wic} \textsc{J. Wichmann}, Bounded approximate units and bounded approximate identities, {\it Proc. Amer. Math. Soc.} 41 (1973) 547–550.
\end{thebibliography}
\end{document}